\documentclass[12pt,a4paper]{amsart}
\pdfoutput=1

\usepackage{amsmath,amssymb,amsthm}  
\usepackage{mathtools}
\usepackage{tikz}

\usepackage{xcolor} 	
\usepackage{hyperref}
\hypersetup{
	colorlinks,
    linkcolor={red!60!black},
    citecolor={green!60!black},
    urlcolor={blue!60!black},
}
\usepackage[abbrev, msc-links]{amsrefs} 

\usepackage[utf8]{inputenc}
\usepackage[T1]{fontenc}
\usepackage{lmodern}
\usepackage[babel]{microtype}
\usepackage[english]{babel}

\usepackage{geometry}
\usepackage{enumitem}

\let\polishlcross=\l
\def\l{\ifmmode\ell\else\polishlcross\fi}

\let\emptyset=\varnothing

\let\theta=\vartheta
\let\rho=\varrho
\let\phi=\varphi

\def\NN{\mathbb N}

\DeclareFontFamily{U}  {MnSymbolC}{}
\DeclareSymbolFont{MnSyC}         {U}  {MnSymbolC}{m}{n}
\DeclareFontShape{U}{MnSymbolC}{m}{n}{
    <-6>  MnSymbolC5
   <6-7>  MnSymbolC6
   <7-8>  MnSymbolC7
   <8-9>  MnSymbolC8
   <9-10> MnSymbolC9
  <10-12> MnSymbolC10
  <12->   MnSymbolC12}{}
\DeclareMathSymbol{\powerset}{\mathord}{MnSyC}{180}

\theoremstyle{plain}
\newtheorem{thm}{Theorem}[section]
\newtheorem{fact}[thm]{Fact}

\newtheorem{cor}[thm]{Corollary}
\newtheorem{lemma}[thm]{Lemma}

\newtheorem{exmp}[thm]{Example}
\newtheorem{quest}[thm]{Question}
\newtheorem{conj}[thm]{Conjecture}

\theoremstyle{definition}

\newtheorem{dfn}[thm]{Definition}

\DeclareMathOperator{\Cl}{Cl}

\DeclareMathOperator{\Fr}{Fr}
\DeclareMathOperator{\Ins}{On}
\DeclareMathOperator{\Reach}{R}

\title{Circuits through prescribed edges}

\author{Paul Knappe}
\author{Max Pitz}
\address{University of Hamburg, Department of Mathematics, Bundesstra{\ss}e 55 (Geomatikum), 20146 Hamburg, Germany}
\email{paul.knappe@studium.uni-hamburg.de, max.pitz@uni-hamburg.de}

\keywords{circuits; closed trails; Eulerian subgraph.}
\subjclass[2010]{05C38, 05C40, 05C45} 

\begin{document}

\begin{abstract}
We prove that a connected graph contains a circuit---a closed walk that repeats no edges---through any $k$ prescribed edges if and only if it contains no odd cut of size at most~$k$.
\end{abstract}

\maketitle

\section{Introduction}\label{sec:introduction}
 
Finding a cycle\footnote{This paper follows the notation in Bollob\'as' \emph{Graph theory}, \cite{bollobasintroductorycourse1979}.} containing certain prescribed vertices or edges of a graph is a classical problem in graph theory. When specifying vertices, already Dirac \cite{dirac1960abstrakten}*{Satz~9} observed that, in a $k$-connected graph, any $k$ vertices lie on a common cycle, and that this is not necessarily true for $k+1$ distinct vertices. 
Dirac's results marked the starting point for a number of results giving conditions under which a set of vertices lies on a common cycle, and we refer the reader to Gould's survey \cite{gould2009look} for a detailed overview of results in this direction.

When trying to find a cycle containing some specified edges, research has been driven by a number of conjectures due to Lov\'asz~\cite{Lovasz} (1973) and Woodall~\cite{MR0439686} (1977). The strongest of these is the following:

\begin{conj}[Lov\'asz-Woodall Conjecture]\label{conj:lovaszwoodall}
	Let $S$ be a set of $k$ independent edges in a $k$-connected graph $G$. If $k$ is even or $G-S$ is connected, then there is a cycle in $G$ containing $S$.
\end{conj}

Building on earlier work by Woodall, in particular on a technique of Woodall from \cite{MR0439686} called the \emph{Hopping Lemma}, H\"aggkvist and Thomassen~\cite{MR676859} (1982) and Kawarabayashi~\cite{MR1877899}*{Theorem 2} (2002) established the following variants of the Lov\'asz-Woodall Conjecture. First, in the case of H\"aggkvist and Thomassen, by setting out from the stronger assumption of $(k+1)$-connectedness,  and second, in the case of Kawarabayashi, by obtaining a weaker conclusion, namely, two cycles instead of one.
 
 \begin{thm}[H\"aggkvist and Thomassen]
 \label{conj:woodall}
	For any set $S$ of $k$ independent edges in a $(k+1)$-connected graph, there is a cycle in $G$ containing $S$.
\end{thm}

\begin{thm}[Kawarabayashi]\label{thm:kawarabayashi}
	Let $S$ be a set of $k$ independent edges in a $k$-connected graph $G$. If $k$ is even or $G-S$ is connected, then $S$ is contained in one or a union of two vertex disjoint cycles of $G$.
\end{thm}

In the present paper, we are interested in a further variant of the problem, where instead of a cycle we aim to find a circuit---a closed walk that repeats no edges (but may repeat vertices)---containing a set of prescribed edges. Clearly, for this variant, it is no longer necessary to assume our edges to be independent. If one aims for results similar in spirit to the cycle case above, it seems natural to consider edge-connectivity instead of vertex connectivity. But whereas in the above cases, vertex connectivity is a far-from necessary condition, the corresponding version for circuits admits a complete characterisation in terms of edge cuts, which is the main result of our paper.

\begin{thm}\label{thm:characterisation of Graphs containing a cycle through prescribed edges}
A connected graph contains a circuit through any $k$ prescribed edges if and only if it contains no odd cut of size at most~$k$.
\end{thm}

\begin{cor}\label{cor:2k-1-property equivalent to 2k-property} If for some $k \in \NN$ a connected graph contains a circuit through any $2k-1$ prescribed edges, then it also contains a circuit through any $2k$ prescribed edges.
\end{cor}

While all the graphs treated in this paper are simple, one can easily derive the same characterisation for multigraphs, since subdividing every edge of a multigraph once does not give rise to new odd cuts.

To see that the condition in Theorem~\ref{thm:characterisation of Graphs containing a cycle through prescribed edges} is necessary, recall that the graph given by the vertices and edges of a circuit is Eulerian, i.e.\ even and connected, and so a necessary requirement for finding a circuit through a set of edges is that it can be extended to an even subgraph. The latter has been characterised by Jaeger~\cite{MR519177} in 1979.

\begin{thm}[Jaeger]\label{thm:jaeger}
	A set of edges in a graph $G$ is contained in an even subgraph of $G$ if and only if it contains no odd cut of $G$.
\end{thm}

However, while Jaeger's theorem immediately shows the necessity of our characterising condition in Theorem~\ref{thm:characterisation of Graphs containing a cycle through prescribed edges}, it does not yield its sufficiency, as Jaeger's even subgraph is not necessarily connected (even if $G$ is). This issue was also overlooked by Lai~\cite{MR1812342}.
See Section~\ref{sec:Concluding remarks and an open question} for further discussion when Jaeger's condition does give rise to a circuit.

\begin{exmp}[Counterexample to \cite{MR1812342}*{Theorem 1.1 \& 4.1}]
\label{ex:jaeger}
Let $k \geq 3$, let $G$ be the ladder with $k+1$ rungs, and $S$ be a set of rungs of $G$ of size $3 \leq |S| \leq k$. Then $S$ extends to an even subgraph of $G$, but every such even subgraph has at least $\lceil \frac{|S|}{2}\rceil \geq 2$ components.
\end{exmp}

\begin{figure}[ht]
	\centering
\begin{tikzpicture}[scale=1]

\definecolor{colore}{RGB}{255,0,0}
\definecolor{colorC}{RGB}{1,41,150}

\pgfmathsetmacro\r{2}
\pgfmathsetmacro\da{4}
\pgfmathsetmacro\db{\da+2}
\pgfmathsetmacro\dc{\da+1}
\pgfmathsetmacro\captdiste{5}
\pgfmathsetmacro\captdistC{2.5}

\foreach \s in {0,...,3}
{
\node [draw, circle,scale=.3, fill] (nodeu\s) at (\r*\s, 0.5*\r) {};
\node [draw, circle,scale=.3, fill] (nodel\s) at (\r*\s,-0.5*\r) {};
}

\node [draw, circle,scale=.3, fill] (nodeu4) at (\r*\db, 0.5*\r) {};
\node [draw, circle,scale=.3, fill] (nodel4) at (\r*\db,-0.5*\r) {};

\draw (nodeu1) -- (nodel1); : 

\draw (nodeu0) -- (nodel0);

\foreach \s in {1,...,3}
{
\pgfmathsetmacro\t{\s-1}
\draw (nodeu\t) -- (nodeu\s);
\draw (nodel\t) -- (nodel\s);
\draw [colore, thick] (nodeu\s) -- node [left= \captdiste pt]{$e_{\s}$} (nodel\s);
}

\draw (nodeu3) -- (\r*\da-\r/3, 0.5*\r);
\draw (nodel3) -- (\r*\da-\r/3,-0.5*\r);

\draw (nodeu4) -- (\r*\dc+\r/3, 0.5*\r);
\draw (nodel4) -- (\r*\dc+\r/3,-0.5*\r);

\draw [colore, thick] (nodeu4) -- node [left=\captdiste pt]{$e_{k}$} (nodel4);

\foreach \s in {0,...,2}
{
\node[draw, circle,scale=.1, fill] (dotu\s) at (\r*\da+\r*0.5*\s,0.5*\r) {};
\node[draw, circle,scale=.15, fill, colore] (dotm\s) at (\r*\da+\r*0.5*\s,0) {};
\node[draw, circle,scale=.1, fill] (dotm\s) at (\r*\da+\r*0.5*\s,-0.5*\r) {};

\draw [colorC, thick] (\r*1.25, 0.75*\r) to [out=300, in=420] node [left=\captdistC pt]{$C$} (\r*1.25, -0.75*\r);
\draw [colorC, thick] (\r*2.25, 0.75*\r) to [out=300, in=420] node [left=\captdistC pt]{$C'$} (\r*2.25, -0.75*\r);
}
\end{tikzpicture}

\caption{A ladder with specified rungs $S=\{e_1,\dots,e_k\}$.}
\label{fig:counterexamplehong}
\end{figure}
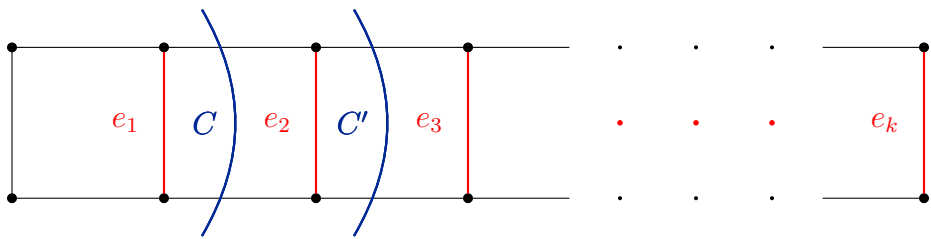

\begin{proof}
    Since $G-S$ is connected, the set $S$ does not contain any cut of $G$ (regardless of its parity), and so $S$ extends to an even subgraph by Theorem~\ref{thm:jaeger}.
    
    Now let $e_1,e_2,e_3 \in S$ be three edges ordered from left to right (cf. Figure~\ref{fig:counterexamplehong}), and suppose for a contradiction there is an even, connected subgraph $H$ of $G$ containing $e_1,e_2,e_3$. Let $C$ and $C'$ be the edge cuts consisting of the two incident edges to the left and to the right of $e_2$ respectively (cf. Figure \ref{fig:counterexamplehong}). Since $H$ is connected and contains $e_1$ and $e_3$, $H$ meets both cuts $C$ and $C'$. Since $H$ is even, it meets every cut of $G$ in an even number of edges, and so $C \cup C' \subset E(H)$. But then both end vertices of $e_2$ have degree three in $H$, a contradiction to $H$ being even.
    
    In particular, if $H$ is an even subgraph containing $S$, then every component of $H$ contains at most two rungs from $S$, and so $H$ has at least $\lceil \frac{|S|}{2}\rceil$ components.
  \end{proof}
  
So instead of referring to Jaeger's theorem for proving the sufficiency of the characterising condition in Theorem~\ref{thm:characterisation of Graphs containing a cycle through prescribed edges}, we once more build on the technique of Woodall's hopping lemma.

Finally, let us mention the survey by Catlin~\cite{catlin1992supereulerian} for related research on the existence of spanning circuits in a graph. Lai~\cite{MR1812342}*{Theorem 3.3} established the following sufficient condition for a graph to contain a spanning circuit through any $k$ prescribed edges:

\begin{thm}[Lai]
\label{thm:lai}
For $k \in \NN$ let $f(k)$ be the smallest even integer $\geq \max(k,4)$. If $G$ is $f(k)$-edge-connected, then $G$ contains a spanning circuit through any $k$ prescribed edges.
\end{thm}

A related variant is to find spanning trails (not necessarily closed) containing a given set of edges, see e.g.\  \cite{xu2014spanning} and the references therein. 

\section{Preliminaries}\label{sec:Preliminaries}
All graphs in this paper are finite and simple. We let $\NN = \{0,1,2,\ldots \}$ and use $[n]=\{1,2,\ldots,n \}$ and $[0,n] = \{0,1,\ldots,n \}$.
For our use of the terms \emph{cycle}, \emph{walk}, \emph{trail} and \emph{circuit}, we follow \cite{bollobasintroductorycourse1979}. Let us clarify the use of technical terms now. 

\begin{dfn} \label{dfn:Basics2}
	 Let $G=(V,E)$ be a graph. For a set of vertices $A\subseteq V$, we write 
	 \begin{itemize}
	     \item $\partial_G A := \{uv \in E \colon u \in A, v \notin A \}$ for the \emph{edge boundary} of $A$ in $G$.
	 \end{itemize}
	
	For $F\subseteq E$, we call

\begin{itemize}
	\item $F$ a \emph{cut} of $G$, if there is an $A\subseteq V$ such that $\partial_G A =F$, and
	\item a cut $F$ \emph{odd}, if $|F|$ is odd. Otherwise, we call $F$ \emph{even}.
\end{itemize}
\end{dfn}
Recall that all cuts of some graph are even if and only if all its vertices have even degree. 

\begin{dfn}\label{dfn:Basics1} Let $G=(V,E)$ be a graph, and let $T=v_0\dots v_r$ a walk in $G$. 

\begin{itemize}
	\item $T$ is a \emph{trail} in $G$, if all of its edges are distinct. Further, $v_0$ is the \emph{start vertex} and $v_r$ is the \emph{end vertex} of $T$, and all other vertices are called \emph{inner vertices} of $T$.
	\item $T$ is \emph{closed}, if its start and end vertex agree. A closed trail is also called \emph{circuit}.
    \item $V(T)$ and $E(T)$ denote the vertices and edges of the underlying subgraph of $T$.
\end{itemize}

\end{dfn}

\begin{dfn}\label{dfn:properties of trails}
Let $G=(V,E)$ be a graph. For $x,y\in V$, $X,Y\subseteq V$ and trails $P=p_0\dots p_r$ and $Q=q_0\dots q_w$ in $G$, we define
\begin{itemize}
    \item $PQ$ or $p_0Pp_rQq_w$ is the \emph{concatenated trail} $p_0\ldots p_r q_1\ldots q_w$ (only when $P$ and $Q$ are edge-disjoint and $p_r=q_0$), 
	\item $P$ is an \emph{$X{-}Y$ trail}, if $p_0 \in X$, $p_r\in Y$ and no inner vertex is in $X$ or $Y$. For singletons write $x{-}y$ trail instead of $\{x\}{-}\{y\}$ trail,
	\item $P$ is a \emph{subtrail} of $Q$ with \emph{witnessing interval} $I_P= \{ t_P,\dots t_P+r \}\subseteq[0,w]$, if $p_h=q_{t_P+h}$ for every $h\in[0,r]$ or $p_h=q_{t_P+r-h}$ for every $h\in[0,r]$, and
	\item $\bar{Q}=q_w\dots q_0$ is the \emph{reversed trail} of $Q$.
\end{itemize}
\end{dfn}

\begin{fact}\label{fact:item:unique witnessing interval}
If $P$ is a subtrail of $Q$ and $P$ uses at least one edge, then the witnessing interval $I_P$ of $P$ in $Q$ is unique.
\end{fact}

\begin{proof}
Let $P=p_0\dots p_r$ and $Q= q_0 \dots q_w$ with $r\geq 1$. Note that while for a single vertex $p_i$ there might be several $q_j$ with $p_i = q_j$, for every edge $p_{i-1}p_{i}$ there is a unique $j=j(i) \in [w]$ with  $p_{i-1}p_{i} = h_{j-1}h_{j}$ (since our graphs are simple). From this, it follows that $I_P=\bigcup_{i \in [r]} \{j(i)-1,j(i)\}$, and so the witnessing interval $I_P$ of $P$ in $Q$ is unique.
\end{proof}

\begin{dfn}\label{dfn:intervals}
 Let $(X,<_X)$ be a finite linear order. For $a \leq_X b\in X$, we define
 \begin{itemize}
     \item $[a,b]_{<_X}:=\big\{\l\in X\colon a\leq_X \l\leq_X b  \}$ as the \emph{closed interval} from $a$ to $b$.
 \end{itemize} 
 Further, for a subset $Y \subseteq X$, we write 
 \begin{itemize}
     \item $\max_{<_X} Y$ for the \emph{greatest element} of $Y$ with respect to $<_X$, and
     \item $\min_{<_X} Y$ for the \emph{smallest element} of $Y$ with respect to $<_X$. 
      \end{itemize} 
\end{dfn}

\section{A reduction to the bridge case}
\label{sec:characterisation theorem}

The proof of our characterisation theorem of graphs containing a circuit through any $k$ prescribed edges will proceed via induction on $k$. For the induction step, suppose we have $k+1$ edges $e_1,\dots,\, e_{k+1}$ of $G$ and may assume inductively that any $k$ edges lie on a common circuit in $G$. Let $H$ be such a circuit through $e_1,\dots,e_k$ in $G$. Our task is then to also incorporate the last edge $e_{k+1}$ into a circuit.

As our first result, we will show that it suffices to consider the case where $e_{k+1}$ is a bridge in $G-E(H)$. 
More precisely, we claim that it suffices to prove the following theorem:

\begin{thm}
\label{thm:wlgobridge}
Let $G$ be a graph containing no odd cut of size at most $k+1$, let $\{e_1,\ldots,e_{k+1}\}$ be a collection of $k+1$ edges in $G$, and $H$ be a circuit in $G$ through $e_1,\ldots,e_k$ such that $e_{k+1}$ is a bridge in $G-E(H)$.

Then there exists a circuit $H'$ in $G$ through $e_1,\ldots,e_{k+1}$. Moreover, if an end vertex of $e_{k+1}$ is not in $V(H)$, then we may assume that $H'$ passes it exactly once.
\end{thm}

We defer the proof of Theorem~\ref{thm:wlgobridge} until the next section, and first show how to complete the proof of the Characterisation Theorem~\ref{thm:characterisation of Graphs containing a cycle through prescribed edges} given Theorem~\ref{thm:wlgobridge}.

\begin{proof}[Proof of Theorem~\ref{thm:characterisation of Graphs containing a cycle through prescribed edges} given Theorem~\ref{thm:wlgobridge}]

As announced, the proof of the sufficiency of the characterisation in Theorem~\ref{thm:characterisation of Graphs containing a cycle through prescribed edges} will go via induction on $k$. The base case is easy: A connected graph without odd cuts of size at most $k=1$ is evidently the same as a bridgeless connected graph. But any edge in such a graph lies on a circuit. 

Now assume inductively that Theorem~\ref{thm:characterisation of Graphs containing a cycle through prescribed edges} holds for some integer $k \in \NN$. To prove Theorem~\ref{thm:characterisation of Graphs containing a cycle through prescribed edges} in the case $k+1$, let $G$ be a graph containing no odd cut of size at most $k+1$, and $S=\{e_1,\ldots,e_{k+1}\}$ a collection of $k+1$ edges in $G$. By induction, we may find a circuit $H$ in $G$ through $e_1,\ldots,e_k$. If $e_{k+1} \in E(H)$, we are done.

So assume that $e_{k+1} \notin E(H)$. If $e_{k+1}$ is a bridge of $G-E(H)$, then we are done by Theorem~\ref{thm:wlgobridge} (the moreover-part is not needed in this case). Otherwise, $e_{k+1}$ is not a bridge in $G-E(H)$, and we may pick $D$ as the maximal $2$-edge-connected subgraph of $G-E(H)$ containing~$e_{k+1}$.

Note that $D$ and $H$ are edge-disjoint, but might share vertices. If they do, choose $v\in V(D)\cap V(H)$ arbitrarily. To see that there is a circuit $H^*$ in $D$ containing $v$ and  $e_{k+1}$, construct an auxiliary graph $D'$ from $D$ by subdividing $e_{k+1}$ by a new vertex $w$. Since $D$ is 2-edge-connected, so is $D'$. By  Menger's theorem, there are two edge-disjoint $w{-}v$ paths in $D'$ translating to the desired circuit $H^*$ in $D$. Since $H$ and $H^*$ are edge-disjoint and intersect in $v$, it is clear that $E(H) \cup E(H^*)$ is the edge set of a circuit covering $S$.

Thus, we may assume that $V(D)\cap V(H)=\emptyset$. Let $F:=\partial_G(V(D))\subseteq E\setminus E(H)$ and observe that every edge in $F$ is a bridge in $G-E(H)$. Since $G$ is connected, $F$ is non-empty, and we choose $e_F\in F$ arbitrarily. Write $e_F=uw$ with $u\in V(D)$. Next, we contract $D$ in $G$. Let $G'$ be the resulting graph and $v_D\in V(G')$ be the vertex corresponding to the contracted $D$. 

Observe that $H$ is still a circuit through $e_1,\dots,e_k$ in $G'$, that $v_D$ is not contained in $V(H)$ and that $G'$ is simple. Furthermore, every cut of $G'$ is also a cut in $G$ (after uncontracting $v_D$), and so $G'$ contains no odd cut of size at most $k+1$. Hence, we may apply Theorem~\ref{thm:wlgobridge} to $G'$, $H$ and $e_F$ to find a circuit $H'\subseteq G'$ through $e_1,\dots,e_{k}$ and $e_F$, such that $H'$ passes $v_D$ exactly once (by the moreover-part). Let $e=u'w'$ with $u'\in V(D)$ be the edge in $F$ corresponding to the other edge in $H'$ incident with $v_D$. The circuit $H'$ in $G'$ corresponds to an $u'{-}u$ trail $H^*$ in $G-E(D)$. By subdividing $e_{k+1}$ in $D$ once and using Menger's theorem in the resulting $2$-edge-connected graph $D'$, we find an $u {-} u'$ trail $Q$ in $D$ trough $e_{k+1}$. Since $Q$ and $H^*$ are edge-disjoint, it follows that $uQu'H^*u$ is the desired circuit in $G$ through $e_1,\ldots,e_{k+1}$.
\end{proof}

\section{Proving the bridge case}

In this section, we prove Theorem~\ref{thm:wlgobridge}, completing the proof of the characterisation stated in Theorem~\ref{thm:characterisation of Graphs containing a cycle through prescribed edges}. As indicated in the introduction, our proof of Theorem~\ref{thm:wlgobridge} is based on the so-called Hopping Lemma due to Woodall~\cite{MR0439686}.

Throughout this section, when describing our set-up and stating our auxiliary results, we work in a fixed $2$-edge connected graph $G=(V,E)$, with $S=\{e_1,\ldots,e_{k+1}\}$ a collection of $k+1$ edges of $G$, and $H$ a shortest circuit through $e_1,\ldots,e_k$ in $G$. Any remaining assumptions featuring in Theorem~\ref{thm:wlgobridge} will only be used in the final proof of Theorem~\ref{thm:wlgobridge} itself at the very end of this section.

If $e_1,\dots,e_k$ lie on a cycle $C$, then $C-\{e_1,\dots,e_k\}$ naturally falls apart into components, each of which is a path. If as in our situation $e_1,\dots,e_k$ lie on a common circuit $H$, then $H-\{e_1,\dots,e_k\}$ also falls apart into segments: subtrails $H_1,\ldots,H_k$ of $H$ such that (after relabelling our edges) we have $H=H_1e_1H_2e_2 \ldots e_{k-1}H_ke_k$. Note, however, that different segments of $H-\{e_1,\dots,e_k\}$ are no longer vertex-disjoint (and so do not correspond to components of the \emph{subgraph} $H-\{e_1,\dots,e_k\}$, cf.\ Figure~\ref{fig:segment}).
 
\begin{dfn}\label{dfn:segment}\label{dfn:induced order on segments}
Given the circuit $H=H_1e_1H_2e_2 \ldots e_{k-1}H_ke_k$, we call $H_j$ the \emph{$j$-th segment} of $H$. Since $H$ is  shortest possible, every segment $H_j$ is a path. We let $<_j$ denote the path order on $V(H_j)$ induced by the circuit $H$. 

\begin{figure}[ht]
\centering

\begin{tikzpicture}[scale=1]

\definecolor{colore}{RGB}{255,0,0}
\definecolor{colorH1}{RGB}{92,65,93}
\definecolor{colorH3}{RGB}{236,117,5}
\definecolor{colorH2}{RGB}{1,41,150}

\pgfmathsetmacro\r{2}
\pgfmathsetmacro\angle{27}
\pgfmathsetmacro\captdist{5}

\foreach \s in {0,...,5}
{
\node [draw, circle,scale=.3, fill] (nodeu\s) at (\r*\s, 0.5*\r) {};
\node [draw, circle,scale=.3, fill] (nodel\s) at (\r*\s,-0.5*\r) {};
}

\node [draw, circle,scale=.3, fill] (nodem2) at (\r*2, 0*\r) {};

\foreach \s in {3,...,4}
{
\node [draw, circle,scale=.3, fill] (nodeuu\s) at (\r*\s, 1*\r) {};
\node [draw, circle,scale=.3, fill] (nodell\s) at (\r*\s,-1*\r) {};
}

\draw [colore] (nodeu0) to node [rotate=90]{$>$} node [left=\captdist pt]{$e_3$} (nodel0);

\draw [colorH1] (nodeu0) to node {$>$} node [below=\captdist pt]{$H_1$} (nodeu1);
\draw [colorH3] (nodel0) to node {$<$} (nodel1);

\draw [colorH1] (nodeu1) to node [rotate=180-\angle]{$<$} (nodem2);
\draw [colorH3] (nodel1) to node [rotate=\angle]{$<$} (nodem2);

\draw [colorH2] (nodem2) to node [rotate=\angle]{$<$} (nodeu3);
\draw [colorH1] (nodem2) to node [rotate=180-\angle]{$<$} (nodel3);

\draw [colorH1] (nodel3) to node {$>$} (nodel4);
\draw [colorH1] (nodel4) to node {$>$} (nodel5);

\draw [colorH2] (nodeu3) to node {$<$} (nodeu4);
\draw [colorH2] (nodeu4) to node {$<$} node [below=\captdist pt]{$H_2$} (nodeu5);

\draw [colore] (nodeu5) to node [rotate=90]{$>$} node [right=\captdist pt]{$e_1$}(nodel5);

\draw [colorH3] (nodem2) to node [rotate=90]{$<$} (nodeu2);
\draw [colorH2] (nodem2) to node [rotate=90]{$<$} (nodel2);

\draw [colorH3] (nodeu2) to node [rotate=\angle]{$<$} (nodeuu3);
\draw [colorH2] (nodel2) to node [rotate=180-\angle]{$<$} (nodell3);

\draw [colorH3] (nodeuu3) to node {$<$} node [below=\captdist pt]{$H_3$}(nodeuu4);
\draw [colore] (nodell3) to node {$>$} node [above=\captdist pt]{$e_2$} (nodell4);

\draw [colorH3] (nodeuu4) to node [rotate=90]{$>$} (nodeu4);
\draw [colorH3] (nodeu4) to node [rotate=90]{$>$} (nodel4);
\draw [colorH3] (nodel4) to node [rotate=90]{$>$} (nodell4);

\end{tikzpicture}

\caption{A circuit $H=H_1e_1H_2e_2H_3e_3$ with segments $H_1,H_2,H_3$.}
\label{fig:segment}
\end{figure}
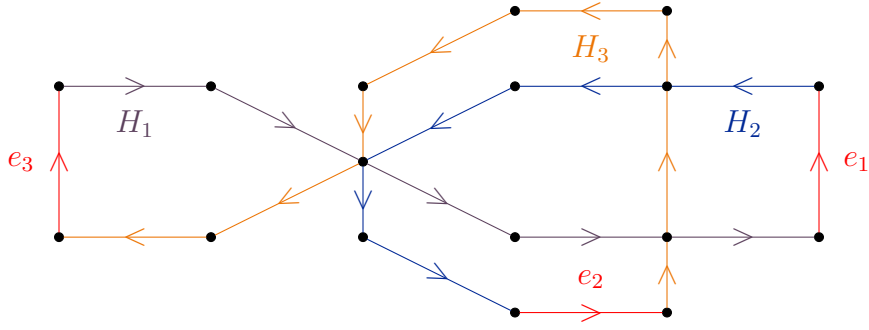
\end{dfn}

\begin{dfn}\label{dfn:inside, closure and frontier}
Given the circuit $H$ with segments $\{H_j\colon j \in [k]\}$, for $U\subseteq V$ and $j\in[k]$, we define (cf.\ Definition~\ref{dfn:intervals})

\begin{enumerate} [label=($\arabic*$)]
	\item \label{dfn:item:inside_j} $\Ins_j(U) := U\cap V(H_j)$ as the \emph{vertices of $U$ on the $j$-th segment of $H$},	
	\item \label{dfn:item:closure_j} $\Cl_j(U) := [\min_{<_j} \Ins_j(U), \max_{<_j} \Ins_j(U)]_{<_j}$ as the \emph{closure of $U$ on the $j$-th segment of $H$},
	\item \label{dfn:item:closure} $\Cl(U) :=  \bigcup_{\l\in [k]} \Cl_{\l}(U)$ as the \emph{closure of $U$ in $H$},

	\item \label{dfn:item:frontier_j} $\Fr_j(U) :=  \{ \min_{<_j} \Ins_j(U), \max_{<_j} \Ins_j(U)  \} $ as the \emph{frontier of $U$ on the $j$-th segment of $H$} and
	\item \label{dfn:item:frontier} $\Fr(U) :=  \bigcup_{\l\in [k]}  \Fr_{\l}(U) $ as the \emph{frontier of $U$ in $H$}.
\end{enumerate}
\end{dfn}

Note that due to the fact that different segments can intersect, the set inclusions $\Cl(U) \subseteq \Cl(\Cl(U))$, $\Cl_j(U) \subseteq \Ins_j(\Cl(U))$ and $\Fr_j(U) \subseteq \Ins_j(\Fr(U))$ might be proper.

\begin{fact}\label{fact:subtrail} For $j\in[k]$ and $U\subseteq V$, we have $\Cl_j(U)$ is a subtrail of $H_j$.
\end{fact}

\begin{dfn}\label{dfn:properties of reach}
For $x,y\in V(G)$ and $X\subseteq V$, we say
\begin{enumerate}[label=($\arabic*$)]
	\item\label{item:admissible} an $x{-}y$ trail $P$ is \emph{admissible}, if it is in $G-E(H)-e_{k+1}$ and $V(P)\cap V(H)\subseteq  \{ x,y  \}$, and
	\item $\Reach(X) :=  \{y'\in V(H)\colon \exists x'\in X\,\exists $ admissible $x' {-} y'$ trail$ \}$ as \emph{reach of $X$ after $H$}.
\end{enumerate}
We stress that the inner vertices of an admissible $x {-} y$ trail are not in $V(H)$.
\end{dfn}

\begin{dfn}\label{dfn:A and B}
 We define an increasing sequence $(A_i)_{i\in\NN}$ recursively by

\begin{enumerate} [label=($\arabic*$)]
	\item $A_0:=\emptyset$,
	\item \label{dfn:item: A_0} $A_1:=\Reach( \{a  \})$, and
	\item \label{dfn:item: A_n} if $A_{i}$ is already defined for some $i\geq 1$, then  $A_{i+1}:=\Reach(\Cl(A_{i}))$.
\end{enumerate}
Further, we set $A:= \bigcup_{i\in\NN}A_i$. Analogously, we define an increasing sequence $(B_i)_{i\in\NN}$ and $B$ by interchanging $a$ with $b$.
\end{dfn}

The idea behind this definition is the simple observation that if $A_1$ and $B_1$ intersect the same segment of $H$, then we clearly would be done. This will not always be possible, and so we iterate this procedure again and again, until we do find one vertex in $A$ and one vertex in $B$ that are contained in the same segment of $H$, as Lemma~\ref{lemma:existence of a starting segment} below shows.

We remark that Definition~\ref{dfn:A and B} of $(A_i)_{i\in\NN}$ differs from Woodall's in that Woodall's admissible paths (see $x \star y$ in \cite{MR0439686}) from $A_{i}$ to new vertices of $A_{i+1}$ are not allowed to start from the frontier of $A_{i}$. 

\begin{lemma}\label{lemma:existence of a starting segment}
	If $\Ins_j(A)=\emptyset$ or $\Ins_j(B)=\emptyset$ for every $j\in[k]$, then $G$ contains an odd cut of size at most $k+1$.
\end{lemma}

\begin{proof}
    First of all, since $G$ is $2$-edge-connected, both $A$ and $B$ are non-empty: Since $G-e_{k+1}$ is connected, any $a{-}V(H)$ path in $G-e_{k+1}$ is an admissible trail which witnesses the non-emptiness of $A_1\subset A$, and similarly for $B$.
    
	Since $A,B \subseteq V(H)$ and $\Ins_j(A)=\emptyset$ or $\Ins_j(B)=\emptyset$ for every $j\in[k]$, $A$ and $B$ are disjoint. Further, from the pigeonhole principle it  follows without loss of generality, that $  | \{j\in[k]\colon \Ins_j(A)\neq\emptyset  \}  |\leq \left\lfloor \frac{k}{2} \right\rfloor$. Then $$|\partial_H A |= | \bigcup_{j\in [k]} \partial_{(e_{j-1}H_j e_j)} \Ins_j(A) | \leq \sum_{j\in [k]} |\partial_{(e_{j-1}H_j e_j)} \Ins_j(A) | = 2 \cdot  \left\lfloor \tfrac{k}{2} \right\rfloor,$$ and since $H$ induces an even subgraph, $|\partial_H A|$ is even. Thus,  $C:=\partial_H A\cup  \{e_{k+1}  \}$ is odd and has size $|C|\leq2\cdot\left\lfloor \frac{k}{2} \right\rfloor +1 \leq k+1$. 
	
	To complete the proof, it remains to show that $C$ is a cut in $G$. For this, we consider 
	$$D = \{v \in V(G) \colon \exists  a' \in A \cup \{a\}\, \exists \text{ admissible $a'{-}v$ trail}\},$$
	and claim that $\partial_{G}D=C$.
		
	To see $C\subseteq \partial_G D$, note that $\partial_H A \subset \partial_G D$ by definition of $A$. For $e_{k+1} \in \partial_G D$, suppose to the contrary that $b\in D$. Then there exists an admissible $b{-}(A\cup\{a\})$ trail $T$. Since $B \neq \emptyset$, $T$ combined with an admissible $b{-}B$ trail witnesses that $A \cap B \neq \emptyset$, a contradiction.
	
	To prove $\partial_G D\subseteq C$, let us suppose for a contradiction that there exists some edge $e=uv~\in~(\partial_G D)~\setminus~C$ with say $u\in D$ and $v \notin D$. Since $u\in D$, we can pick an admissible trail $T$ starting in some $a'\in A\cup \{a\}$ and ending  in $u$. If $e \in E(H)$, then $u\in V(H)$ and thus $u\in A$ by Definition~\ref{dfn:A and B}. Now $e\in \partial_H(A)$, which contradicts $e\notin C$. So, we assume $e\notin E(H)$. If $u\in V(H)$, then $u \in A$ and the trail $uv$ is a witness for $v\in D$. Otherwise, $Tuv$ is a witness. In any case, this contradicts $v\notin D$.
\end{proof}

Now that we know that $A$ and $B$ intersect the same segment $H_j$ of $H$, it is clear that there is a natural trail in $H$ starting at a vertex of $A$, ending at a vertex of $B$, and containing all of the edges $e_1,\ldots,e_k$. If we consider the `first time' that $A_n$ and $B_m$ intersect a given segment $H_j$, then this trail has the following three crucial properties of Definition~\ref{dfn:coherent}, as Lemma~\ref{lemma:trivial coherent trail} shows.

\begin{dfn}\label{dfn:coherent}
 For $n,m\in \NN$, we say a trail $Q=q_0\dots q_{w}$ is \emph{$A_{n}{-}B_{m}{-}$coherent}, if

\begin{enumerate}[label=(C$_\arabic*$)]
	\item\label{itemC1} $e_1,\dots,e_k\in E(Q)$, $q_0\in A_{n+1}$ and $q_w\in B_{m+1}$,
	\item\label{itemC2} for every $s\in[w]$ with $q_{s-1} q_{s}\in E\setminus E(H)$, there exist $r,t\in[0,w]$ with $q_r,q_t\in V(H)$ and $r < s \leq t$ such that $q_r Q q_t$ is an admissible $q_r{-}q_t$ trail and each of the sets $A_{n+1}$ and $B_{m+1}$ contains at most one of $q_r$ and $q_t$, and
	\item\label{itemC3} for every $j\in[k]$, $\Cl_j(A_{n})$ and $\Cl_j(B_{m})$ are subtrails of $Q$ with witnessing intervals $I_{A_{n},j}$ and $I_{B_{m},j}$ such that $I_{X,j}\cap I_{Y,j'} =\emptyset$ for every $X,Y\in \{A_{n},B_{m}\}$ and every two distinct $j \neq j'\in [k]$.
\end{enumerate}
\end{dfn}

\begin{lemma}\label{lemma:trivial coherent trail}
If $\Cl_j(A_{n^*})\neq \emptyset \neq \Cl_j(B_{m^*})$ for some $j \in [k]$, then there exists an $A_{n}{-}B_{m}{-}$ coherent trail for some $n<n^*$, $m<m^*$.
\end{lemma} 

\begin{proof}
	Let $j$ be in $[k]$ such that $\Cl_j(A_{n^*})\neq \emptyset \neq \Cl_j(B_{m^*})$. Choose $n<n^*$ and $m<m^*$ minimal such that $\Cl_j(A_{n+1})\neq \emptyset \neq \Cl_j(B_{m+1})$ and pick $a_{n+1}\in\Ins_j(A_{n+1})$ and $b_{m+1}\in\Ins_j(B_{m+1})$.\footnote{One could make a stronger minimality assumption by choosing $n,m$ minimal so that $\Cl_j(A_n)\neq \emptyset \neq \Cl_j(B_n)$ for some $j \in [k]$. Following the same proof, this gives rise to a trail $Q$ which satisfies the following stronger variant of \ref{itemC3}, namely $I_{X,j}\cap I_{Y,j'} =\emptyset$ for every $X,Y\in \{A_{n},B_{m}\}$ and every two (not necessarily distinct)  $j,j'\in [k]$. However, we do not need this stronger conclusion for the remainder of our proof.
	} We claim that the trail $Q$ with start vertex $a_{n+1}$ and end vertex $b_{m+1}$ along the circuit $H$ through $e_1,\dots,e_k$ and $H_{j'}$ as subtrail for every $j'\in [k]\setminus \{j \}$ is $A_{n}{-}B_{m}{-}$coherent as desired.
	
	Indeed, \ref{itemC1} holds by construction and \ref{itemC2} is an empty condition. Lastly, since $\Cl_j(A_{n})=\emptyset = \Cl_j(B_m)$, and all other segments $H_{j'}$ for $j' \in [k]\setminus \{j\}$ are subtrails of $Q$ with pairwise disjoint witnessing intervals by construction, also \ref{itemC3} holds for $Q$. \qedhere  
\end{proof}

 While conditions~\ref{itemC1} and \ref{itemC2} are straightforward adaptions from Woodall's notion of coherence~\cite{MR0439686}*{\S III} from paths to trails, a word on \ref{itemC3} might be in order. Given the `time-minimal' subtrail $Q$ of $H$ constructed in Lemma~\ref{lemma:trivial coherent trail}, we aim to modify $Q$ while preserving as much structure of $Q$, and hence of $H$, as possible. Since  segments of $H$ may intersect, the correct notion of `structure preserving' is to think about the trail in terms of time: Our initial trail $Q$ constructed in Lemma~\ref{lemma:trivial coherent trail} spends disjoint time intervals to cover the different segments of $H$ that contain vertices from $\Cl(A_n) \cup \Cl(B_n)$. When modifying $Q$, however, we can no longer require to completely cover all these segments. So instead, we only preserve the property that if $T$ and $S$ are subpaths of \emph{distinct} segments $H_j$ and $H_{j'}$ of the form $T \in \{\Cl_j(A_{n}),\Cl_j(B_{m})\}$ and $S \in \{\Cl_{j'}(A_{n}),\Cl_{j'}(B_{m})\}$, then we continue to spend disjoint time intervals to cover $T$ and $S$.

\begin{thm} \label{lemma:existence of desired circuit}
If there exists an $A_{n}{-}B_{m}{-}$coherent trail for some $n,m\in\NN$, then there also exists an $A_{0}{-}B_{0}{-}$coherent trail.
\end{thm}

For the proof, we need two easy lemmas.

\begin{lemma}\label{fact:coherent is decreasing}
Let $n,m\in\NN$ and $Q=q_0\dots q_{w}$ be an $A_{n}{-}B_m{-}$coherent trail.  If $n\geq 1$ and $q_0\in A_{n}$, then $Q$ is $A_{n-1}{-}B_m{-}$coherent, and if $m\geq 1$ and $q_w\in B_m$, then $Q$ is $A_{n}{-}B_{m-1}{-}$coherent. 
\end{lemma}

\begin{proof}
Due to the symmetry of the statements, we just check the conditions for $Q$ being $A_{n-1}{-}B_m{-}$coherent for $n\geq 1$. Property \ref{itemC1} is clear, and \ref{itemC2} is immediate from the fact that $(A_i)_{i\in\NN}$ is an increasing sequence. 

Finally, \ref{itemC3} follows from the fact that since $(A_i)_{i\in\NN}$ is increasing,  $\Cl_j(A_{n-1})$ is a subtrail of $\Cl_j(A_n)$, and hence we have $I_{A_{n-1},j} \subseteq I_{A_n,j}$ for the respective witnessing intervals for all $j \in [k]$. Since $I_{X,j}\cap I_{Y,j'} =\emptyset$ for every $X,Y\in \{A_{n},B_{m}\}$ and every two distinct $j, j'\in [k]$ holds by assumption, it follows that the same holds for every $X,Y\in \{A_{n-1},B_{m}\}$. \qedhere
\end{proof}

\begin{lemma}\label{lemma:P and Q are disjoint}
Let $n,m\in\NN$ and $v\in (\Cl(A_{n})\cup \{a\}) \cup (\Cl(B_{m})\cup \{b\})$. If $Q=q_0\dots q_{w}$ is an $A_{n}{-}B_{m}{-}$coherent trail and $P$ is an admissible $v{-}V(H)$ trail, then $Q$ and $P$ are edge-disjoint.
\end{lemma}

\begin{proof}
By symmetry we may assume that $v \in \Cl(A_{n})\cup \{a\}$. Suppose for a contradiction that $P$ and $Q$ are not edge-disjoint. Choose $s\in[w]$ such that $q_{s-1}q_{s}$ is the first edge of $P$ that is also in $E(Q)$. Since $q_{s-1}q_{s}\in E(P)\subseteq E\setminus E(H)$ by Definition~\ref{dfn:properties of reach}, it follows from property~\ref{itemC2} of $A_{n}{-}B_{m}{-}$coherent that there are $r,t\in[0,w]$ with $r < s \leq t$ and $q_r,q_t\in V(H)$ such that $q_rQq_t$ is an admissible $q_r{-}q_t$ trail and each set $A_{n+1}$ and $B_{m+1}$ contains at most one of $q_r$ and $q_t$. But since $q_{s-1}q_{s}$ is the first edge of $P$ in $E(Q)$, both $vPq_{s-1}\bar{Q}q_r$ and $vPq_{s-1}Qq_t$ are admissible trails witnessing that $q_r,q_t \in A_{n+1}$ (cf.~Definition~\ref{dfn:A and B}\ref{dfn:item: A_n}), a contradiction.
\end{proof}

\begin{proof}[Proof of Theorem~\ref{lemma:existence of desired circuit}]
Let $n,m$ be minimal such that there is an $A_{n}{-}B_{m}{-}$coherent trail $Q=q_0\dots q_{w}$ with start vertex $q_0 = a_{n+1} \in A_{n+1}$ and end vertex $q_w = b_{m+1}\in B_{m+1}$. We claim that $n=m=0$. Otherwise, without loss of generality we may assume $n \geq 1$. By Lemma~\ref{fact:coherent is decreasing} and the minimality assumption, we have $a_{n+1} \in A_{n+1}\setminus A_{n}$. We write $Q$ as $a_{n+1}Q q_c Qq_dQb_{m+1}$ where $c,d\in[0,w]$ are defined as follows:
	\begin{itemize}
	\item[(a)] Since $a_{n+1} \in A_{n+1}\setminus A_{n}$ and by Definition~\ref{dfn:A and B}\ref{dfn:item: A_n} of $A_{n+1}$, there is an $x\in \Cl(A_{n})$ such that there exists an admissible $x{-}a_{n+1}$ trail $P$ (which might be trivial). From Definition~\ref{dfn:inside, closure and frontier}\ref{dfn:item:closure} of the closure it follows that there is an $j \in [k]$ such that $x\in\Cl_j(A_{n})$. By property~\ref{itemC3} of $A_{n}{-}B_{m}{-}$coherent, $\Cl_j(A_{n})$ is a subtrail of $Q$ with witnessing interval $I_{A_{n},j}\subseteq[0,w]$. Now, we choose $d\in I_{A_{n},j}$ as the unique index with $q_d=x$.
	\item[(b)] Next, choose $c:=\max  \{ r\in[0,w]\colon q_r \in  \bigcup_{i\in[n]} \Fr_j(A_i) \land r \leq d  \}$. If $r:=\min I_{A_{n},j}$, then $q_{r} \in \Fr_j(A_{n})$ and obviously $r\leq d$. Hence, $c$ exists.
	\end{itemize}

\begin{figure}[ht]
\centering
\begin{tikzpicture}[scale=1]

\definecolor{colorQ}{RGB}{0,130,90}
\definecolor{colorQ'}{RGB}{236,117,5}
\definecolor{colorP}{RGB}{92,65,93}
\definecolor{colorI}{RGB}{1,41,150}

\pgfmathsetmacro\r{2}

\foreach \s/\capt  in {0/$a_{n+1}$, 1/$q_{\min I_{A_{n},j}}$,2/$q_c$,3/$q_d$,4/$ $,5/$b_{m+1}$}
{
\node [draw, circle,scale=.3, fill, label={[label distance=5pt]90:\capt}] (nodem\s) at (\r*\s, 0) {};
}

\foreach \s/\t/\whereone/\captone/\capttwo/\captthree/\captfour/\captfive/\captsix in {0/1/{}/{$<$}/{}/{}/{}/$[$/{} ,1/2/{}/{$<$}/{}/{}/{}/{}/{},2/3/{}/{}/{}/{}/{}/{}/{},3/4/{}/{$>$}/{$Q'$}/{}/{}/{}/{},4/5/{}/{$>$}/{}/{$Q$}/{$]$}/{}/{$\Cl_j(A_n)$}}
{

\draw [colorQ] (nodem\s) to node [\whereone, text=colorQ']{\captone} node [below=5pt, text=colorQ']{\capttwo} node [below=5pt, text=colorQ]{\captthree} node [at start, text=colorI]{\captfour} node [at end, text=colorI]{\captfive} node [above=5pt, at start, text=colorI]{\captsix}(nodem\t);
}

\draw [colorP] (nodem0) to [out=315, in=225] node [text=colorQ']{$>$} node [below=5pt]{$P$} (nodem3);
 
\end{tikzpicture}

\caption{Obtaining the rerouted trail $Q'$ from $Q$.}
\label{fig:rerouting}
\end{figure}
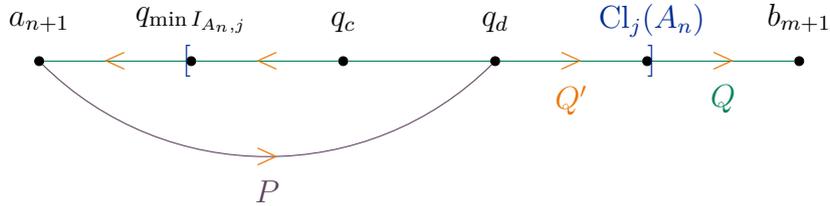

Further, we set $n':=\min  \{i\in [n]\colon q_c\in \Fr_j(A_{i+1})  \}$ and observe
\begin{enumerate}[label=($\arabic*$)]
	\item \label{proof:I_A cap [c,d] empty} $I_{A_{n',j}} \cap [c,d]=\emptyset$, 
	\item \label{proof:I_B cap [c,d] non-empty implies IA} if $I_{B_{m},j}\cap [c,d] \neq \emptyset$, then $\Cl_j(A_{n})\cap\Cl_j(B_{m+1}) \neq \emptyset$,
	\item \label{proof:q_d in B_m implies IA} if $q_d \in B_{m+1}$, then $\Cl_j(A_{n})\cap\Cl_j(B_{m+1}) \neq \emptyset$, and
	\item \label{proof:P and Q are disjoint} $P$ and $Q$ are edge-disjoint.
\end{enumerate}
	
\begin{proof}[Proof of \ref{proof:I_A cap [c,d] empty}]\renewcommand{\qedsymbol}{}
	We assume for a contradiction that $I_{A_{n',j}} \cap [c,d]\neq \emptyset$. Then, either choosing $r$ as ${\min (I_{A_{n',j}}\cap [c,d])}$ or ${\max (I_{A_{n',j}}\cap [c,d])}$ will lead to $q_r \in \Fr_j(A_{n'})$, which is a contradiction to the choice of $c$ or $n'$ because $c\leq r\leq d$. 
\end{proof}
	
\begin{proof}[Proof of \ref{proof:I_B cap [c,d] non-empty implies IA}]\renewcommand{\qedsymbol}{}
	Let 	$I_{B_{m},j}\cap [c,d]\neq\emptyset$. So, $I_{A_{n},j}\cap I_{B_{m},j}\neq\emptyset$ because $[c,d]\subseteq I_{A_{n},j}$. Further, $\Cl_j(A_{n})\cap\Cl_j(B_{m})\subseteq \Cl_j(A_{n})\cap\Cl_j(B_{m+1})$ implies then that $\Cl_j(A_{n})\cap\Cl_j(B_{m+1})\neq\emptyset$. 
\end{proof}	
	
\begin{proof}[Proof of \ref{proof:q_d in B_m implies IA}]\renewcommand{\qedsymbol}{}
	If $q_d \in B_{m+1}$, then, $q_d \in \Cl_j(A_{n}) \cap \Cl_j(B_{m+1}) \neq \emptyset$.
\end{proof}
	
\begin{proof}[Proof of \ref{proof:P and Q are disjoint}]\renewcommand{\qedsymbol}{}
	Since $q_d \in \Cl_j(A_{n})\subseteq\Cl (A_{n})$ and $Q$ is $A_{n}{-}B_{m}{-}$coherent, this follows from Lemma \ref{lemma:P and Q are disjoint}. 
\end{proof}
		
If $I_{B_{m},j} \cap [c,d] \neq \emptyset$ or $q_d \in B_{m+1}$, then \ref{proof:I_B cap [c,d] non-empty implies IA} or \ref{proof:q_d in B_m implies IA} imply that $\Cl_j(A_{n})  \cap \Cl_j(B_{m+1}) \neq \emptyset$, which by Lemma~\ref{lemma:trivial coherent trail} gives rise to a coherent trail that contradicts the minimality of $n$ and~$m$.
Hence, we assume $I_{B_{m},j}\cap [c,d] = \emptyset$ and $q_d \notin B_{m+1}$.  

Now we reroute $Q$ and obtain $Q':=q_c\bar{Q}a_{n+1}\bar{P}q_dQb_{m+1}$, see Figure~\ref{fig:rerouting}. From \ref{proof:P and Q are disjoint} it follows that $Q'$ is a trail. We show that $Q'$ is $A_{n'}{-}B_{m}{-}$coherent, contradicting the minimality of $n$ and $m$:

\begin{itemize}
\item[(C$_1$)] Since $E(q_c\dots q_d)\subseteq E(H_j)$ and since all our edges satisfy $e_{i}\notin E(H_j)$, the fact that $Q$ satisfied \ref{itemC1} implies that $Q'$ uses $e_1,\dots,e_k$. Also, the start vertex $q_c$ is in $\Fr_j(A_{n'+1})\subseteq A_{n'+1}$ and the end vertex $b_{m+1}$ is still in $B_{m+1}$.
\item[(C$_2$)] Because $a_{n+1}\notin A_n\supseteq A_{n'+1}$ and $q_d\notin B_{m+1}$, each of the sets $A_{n'+1}$ and $B_{m+1}$ contains at most the start or the end vertex of $P$. Also, the $q_d{-}a_{n+1}$ trail $P$ is admissible. This implies that \ref{itemC2} is true for edges that are in $P$.
For edges that are not in $P$, it follows directly from $Q$'s \ref{itemC2} and $q_c,q_d \in V(H)$.
\item[(C$_3$)] Due to \ref{proof:I_A cap [c,d] empty} and $I_{B_{m},j} \cap [c,d]=\emptyset$, the trails $\Cl_{j'}(A_{n'})$ and $\Cl_{j'}(B_{m})$ are subtrails of $q_1\dots q_c$ or $q_d\dots q_w$ for every $j'\in[k]$. Hence, $Q'$ inherits property~\ref{itemC3} from $Q$. \qedhere
\end{itemize}
\end{proof}

We are now ready to complete the proof of Theorem~\ref{thm:wlgobridge}.

\begin{proof} [Proof of Theorem~\ref{thm:wlgobridge}]
Since $G$ contains no odd cut of size at most $k+1$, Lemma~\ref{lemma:existence of a starting segment} implies that $\Ins_j(A)\neq\emptyset \neq \Ins_j(B)$ for some $j\in[k]$.
By Lemma~\ref{lemma:trivial coherent trail} there is an $A_{n}{-}B_{m}{-}$coherent trail in $G-e_{k+1}$ for some $n,m \in \NN$, and so by Theorem~\ref{lemma:existence of desired circuit} there also exists an $A_{0}{-}B_{0}{-}$coherent trail $Q$ from a vertex $a_1\in A_1$ to a vertex $b_1 \in B_1$ in $G-e_{k+1}$.

By Definition~\ref{dfn:A and B}\ref{dfn:item: A_0} of $A_1$ and $B_1$, there is an admissible $a{-}a_1$ trail $P_a$ and an admissible $b{-}b_1$ trail $P_b$. Since $e_{k+1}$ is a bridge in $G-E(H)$,\footnote{We remark that this is the only place in our argument where we use that $e_{k+1}$ is a bridge in $G-E(H)$.} the trails $P_a$ and $P_b$ are vertex-disjoint. Thus, $P_a$, $P_b$, $Q$ and $e_{k+1}$ are edge-disjoint by Lemma~\ref{lemma:P and Q are disjoint} and Definition~\ref{dfn:properties of reach}\ref{item:admissible}.
Together with property~\ref{itemC1} of $Q$, it follows that $H':=baP_aa_1Qb_1\bar{P_b}b$ is the desired circuit in $G$ through $e_1,\dots,e_{k+1}$.
	
To see the moreover-part of Theorem~\ref{thm:wlgobridge}, observe that if $a\notin V(H)$, then $a\notin V(Q)$ due to \ref{itemC2} and Definition \ref{dfn:A and B}\ref{dfn:item: A_0} of $A_1$. Thus, the circuit $H'$ passes $a$ once, since $P_a$ and $P_b$ are vertex disjoint. The same holds for $b$.
\end{proof}

\section{Concluding remarks and an open question}\label{sec:Concluding remarks and an open question}

To find a circuit through any $k$ prescribed edges we employed a global property by forbidding all odd cuts of bounded size. However, if we are only interested in one specific edge set, forbidding all bounded sized odd cuts seems unnecessarily strong: For example, if our $k$ edges are contained in a $(k+1)$-edge-connected subgraph, then it is irrelevant whether the whole graph contains some further small odd cuts. Hence, the following natural question arises:

\begin{quest}\label{quest:local condition for extension to a circuit}
When can a given edge set of a graph $G$ be covered by a circuit in $G$?
\end{quest}

One line of investigation could be whether a condition similar to the one in Jaeger's theorem~\ref{thm:jaeger} could be of additional help:

\begin{dfn}\label{thm:strong jaeger's theorem}
    For any $k\in \NN$, let $g(k)$ be the smallest integer such that a set of at most $k$ edges in a $g(k)$-edge-connected graph $G$ is covered by a circuit in $G$ if and only if it contains no odd cut of $G$.
\end{dfn}

\begin{lemma}
\label{lem_inequalities}
For any $k\in \NN$,
    \begin{enumerate}
        \item $g(k)\leq m \leq k+1$, where $m$ is the smallest even integer ${\geq} k$, and
        \item for $k\geq 4$, $g(k) > \l$, where $\l$ is the greatest odd integer ${\leq} \frac{1}{2} (\sqrt{8 k-7}+1)$.
        \end{enumerate}
\end{lemma}
\begin{proof}
The first part follows directly from Theorem~\ref{thm:characterisation of Graphs containing a cycle through prescribed edges}.

For the lower bound of $g(k)$, let $\l$ is the greatest odd integer ${\leq} \frac{1}{2} (\sqrt{8 k-7}+1)$, and consider $H_i$ to be a $K_\l$ with $V(H)= \left\{ v_{i,1},\dots,v_{i,\l} \right\}$ for $i\in [2]$. Further, we define $G:= H_1 + H_2 + \{v_{1,j}v_{2,j} \colon j\in[\l]\}$. We remark that $G$ is $\l$-connected. Now, we pick $S:= E(H_1) \cup \{e\}$ where $e$ is some edge of $E(H_2)$. We calculate $$|S| = \binom{\l}{2}+1 =  \frac{\l (\l-1)}{2} +1 \leq k$$ where the inequality holds for $\l\leq \frac{1}{2} (\sqrt{8 k-7}+1)$.
By Theorem~\ref{thm:jaeger}, $S$ contains no odd cut of $G$, because $S$ is contained in the even subgraph $H_1 + H_2$. But clearly there exists no circuit $H'$ in $G$ that covers $S$.
\end{proof}

\begin{fact}\label{fact:Strong Jaeger's theorem} 
We have
    $$g(1)=0, \; g(2) = 2, \; g(3) = 3 \, \text{ and } g(4) = 4.$$
\end{fact}

\begin{proof}
To see $g(1)=0$, observe that any edge not being a bridge of its component must lie on a cycle.

For $g(2)=2$, note that $g(2) \leq 2$ by Lemma~\ref{lem_inequalities}, and $g(2) > 1$ by considering two disjoint cycles connected by an edge, and letting $S$ consist of one edge from each cycle. 
    
Next, Example~\ref{ex:jaeger} shows $g(3)>2$. For $g(3)\leq 3$, let $G$ be a $3$-edge-connected graph and $S$ be a $3$-set of edges which contains no odd cut of size at most three. By Theorem~\ref{thm:jaeger}, there exists an even subgraph $H$ of $G$. We choose $H$ subgraph-minimal, and so $H$ has at most three components.
    
First, we assume that $H$ has three components $C_1, C_2, C_3$, and reduce it to the case where $H$ has two components by considering the three edge-disjoint $V(C_1){-}V(C_2+C_3)$ paths in $G$ which exist by Menger's theorem. 
    
Now, we assume that $H$ has two components $C_1, C_2$ where without loss of generality $|E(C_1)\cap S|=1$. Again there are three edge-disjoint $V(C_1){-}V(C_2)$ paths in $G$. At least two of them meet the same segment of $C_2$ such that we can construct a cycle in $G$ which goes through all three edges.
    
Finally, $g(4)=4$ follows from Lemma~\ref{lem_inequalities}.
\end{proof}

Thus, by adding Jaeger's condition, for odd $|S|$ it appears we need less edge connectivity than before. It might be an interesting problem to find the precise values for the function $f$, or at least to improve any of the bounds given in Lemma~\ref{lem_inequalities}. In particular, we were not able to find an example witnessing $g(5)>4$.

\bibliography{references}

\end{document}